\documentclass{article}
                            
\usepackage{amsmath}
\usepackage{amssymb}
\usepackage{amsthm}
\usepackage[english]{babel}
\usepackage{geometry}

\newtheorem{teo}{Theorem}[section]
\newtheorem{coro}[teo]{Corollary}
\newtheorem{oss}[teo]{Remark}
\newtheorem{defn}[teo]{Definition}
\newtheorem{lem}[teo]{Lemma}
\newtheorem{pro}[teo]{Proposition}
\newcommand{\B}{\mathbb{B}}
\newcommand{\HH}{\mathbb{H}}
\newcommand{\s}{\mathbb{S}}
\newcommand{\rr}{\mathbb{R}}
\newcommand{\N}{\mathbb{N}}

\DeclareMathOperator{\IIm}{Im}
\DeclareMathOperator{\RRe}{Re}
\DeclareMathOperator{\ext}{ext}
\newcommand{\de}{\partial_C}

\title{\bf A Bloch-Landau Theorem for slice regular functions}

\author{Chiara Della Rocchetta \footnote{The three authors acknowledge the support of G.N.S.A.G.A. of INdAM and MIUR (Research Project ``Variet\`a reali e complesse: geometria, topologia e analisi armonica'')}\\
\normalsize Dipartimento di Matematica ``U. Dini'', Universit\`a di Firenze \\
\normalsize Viale Morgagni 67/A, 50134 Firenze, Italy, mchidel@gmail.com \\
\and Graziano Gentili \\
\normalsize Dipartimento di Matematica ``U. Dini'', Universit\`a di Firenze \\
\normalsize Viale Morgagni 67/A, 50134 Firenze, Italy,  gentili@math.unifi.it \\
\and Giulia Sarfatti\\
\normalsize Dipartimento di Matematica ``U. Dini'', Universit\`a di Firenze \\
\normalsize Viale Morgagni 67/A, 50134 Firenze, Italy,  sarfatti@math.unifi.it \\
}

\date{}
\begin{document}
\maketitle

\abstract{The Bloch-Landau Theorem is one of the basic results in the geometric theory of holomorphic functions. It establishes that the image of the open unit disc $\mathbb{D}$ under a holomorphic function $f$ (such that $f(0)=0$ and $f'(0)=1$) always contains an open disc with radius larger than a universal constant. In this paper we prove a Bloch-Landau type Theorem for slice regular functions over the skew field $\HH$ of quaternions. If $f$ is a regular function on the open unit ball $\B\subset \HH$, then for every $w \in \B$ we define the {\em regular translation} $\tilde f_w$ of $f$. 
 The peculiarities of the non commutative setting lead to the following statement: there exists a universal open set contained in the image of  $\B$ through some regular translation $\tilde f_w$ of any slice regular function $f: \B \to \HH$ (such that $f(0)=0$ and $\de f(0)=1$). For technical reasons, we introduce a new norm on the space of regular functions on open balls centred at the origin, equivalent to the uniform norm, and we investigate its properties.}

\date{}

{\bf keywords:} Functions of a quaternionic variable, Bloch-Landau Theorem.

\section{Introduction}

After the great development of the theory of holomorphic functions, there have been many attempts, and successes of different character, to build analogous theories of regular functions whose domain and range were the quaternions $\HH$. The most successful theory, by far, has been the one due to Fueter (see the foundation papers \cite{fueter1,fueter2}, the nice survey \cite{Sudbery}, the book \cite{libro daniele} and references therein). More recently, Cullen (see \cite{Cullen}) gave a definition of regularity that Gentili and Struppa reinterpreted and developed in \cite{gs, GSAdvances}, giving rise to the rich theory of slice regular functions.  
For this class of functions there have been proved the corresponding of several results of the complex setting, that often assume a different flavour in the quaternionic setting. For instance we can cite, among the most basic, results that concern the Cauchy Representation Formula and the Cauchy kernel, the Identity Principle, the Maximum Modulus Principle, the Open Mapping Theorem and a new notion of analyticity \cite{kernel, open,zeri, power, GSAdvances}.  Moreover, the theory of slice regular functions has been extended and generalized to the Clifford Algebras setting, and to the more general setting of alternative algebras, originating a class of functions that is also under deep investigation at the moment (see \cite{libro2} and references therein, and \cite{ghiloniperotti0, ghiloniperotti1}). 

\noindent An important property, that distinguishes slice regular functions from the Fueter regular ones, is that power series with quaternionic coefficients on the right, $\sum_{n=0}^{\infty}q^na_n$ are slice regular. Furthermore, on open balls centred at the origin, to have a power series expansion is a necessary and sufficient condition for a function to be slice regular. 
Since we will work with functions defined on open balls centred at the origin, we will make use of this characterization. 
Indeed, the purpose of this paper is to prove an analog of the Bloch-Landau Theorem for slice regular functions. In the complex case, this result  is an important fact in the study of the range of holomorphic functions defined on the  open unit disc $\mathbb{D}$. It states that the image of $\mathbb{D}$ through a holomorphic function cannot be ``too much thin''. In fact, under certain  normalizations, it contains always a disk with a universal radius. One of the first lower bounds, $\frac{1}{16}$,  of this constant is due to Landau, see the book \cite{Landau2}. The same author gave also better estimates in \cite{Landau}. Recently, the Bloch's Theorem in the quaternionic setting has been investigated with success in \cite{Morais}.

\noindent In our approach, since composition of slice regular functions is not regular in general, we have to define the set of {\em regular translations} of a regular function $f$ defined on the open unit ball $\mathbb{B}$ of the space of quaternions.
We then prove the existence of a universal open set $\mathcal{O}$, different from a ball, always contained in the image of the open unit ball $\B$ under some regular translation of any (normalized) slice regular function $f$. Thus we provide a further tool to the geometric theory of slice regular functions.

The paper is organized as follows: Section \ref{P} is dedicated to set up the notation and to give the preliminary results; in Section \ref{3} we prove an important property of the uniform norm on balls centred at the origin; in Section \ref{4} we define a new norm on the space of slice regular functions (on open balls centred at the origin) and we use it to set up a mean value theorem; the last section is devoted to prove a Bloch-Landau type Theorem for slice regular functions.

\section{Preliminaries}
\label{P}

Let $\HH$ denote the non commutative real algebra of quaternions with standard basis $\{1,i,j,k\}$. The elements of the basis satisfy the multiplication rules 
\[i^2=j^2=k^2=-1,\; ij=k=-ji,\; jk=i=-kj,\; ki=j=-ik,\] 
that extend by distributivity to all $q= x_0 +x_1i+x_2j+x_3k$ in $\HH$. 
Every element of this form is composed by the {\em real} part $\RRe(q)=x_0$ and the {\em imaginary} part  $\IIm(q)=x_1i +x_2j +x_3k$. The {\em conjugate} of $q\in\HH$ is then $\bar{q}=\RRe(q)-\IIm(q)$ and its {\em modulus} is defined as $|q|^2=q\bar{q}=\RRe(q)^2+|\IIm(q)|^2$. We can therefore calculate the multiplicative inverse of each $q\neq 0$ as $q^{-1}=\frac{\bar{q}}{|q|^2}$.
Notice that for all non real $q\in \HH$, the quantity $\frac{\IIm(q)}{|\IIm(q)|}$ is an imaginary unit, that is a quaternion whose square equals $-1$. Then we can express every $q\in \HH$ as $q=x+yI$, where $x,y$ are real (if $q\in\rr$, then $y=0$) and $I$ is an element of the unit $2$-dimensional sphere of purely imaginary quaternions,
$$\mathbb{S}=\{q\in \HH \ | \ q^{2}=-1 \}.$$
In the sequel, for every $I\in \mathbb{S}$ we will define $L_I$ to be the plane $\mathbb{R}+\mathbb{R}I$, isomorphic to $\mathbb{C}$, and, if $\Omega$ is a subset of $\HH$, by $\Omega_I$ the intersection $\Omega \cap L_I$.
Also, for $R >0$, we will denote the open ball centred at the origin with radius $R$ by
\[
B(0,R)=\{q\in \HH | \ |q|<R \}.
\]
We can now recall the definition of slice regularity.
\begin{defn}
A function $f: B=B(0,R) \rightarrow \HH$ is {\em slice regular} if, for all $I\in \mathbb{S}$,  its restriction $f_I$ to $B_I$  is {\em holomorphic}, that is  it has continuous partial derivatives and it satisfies
$$\overline{\partial_I} f(x+yI):=\frac{1}{2}\Big(\frac{\partial}{\partial x}  +I\frac{\partial}{\partial y}\Big)f_I(x+yI)=0$$
for all $x+yI \in B_I$.
\end{defn}
\noindent In the sequel, we will avoid the prefix slice when referring to slice regular functions.
 
\noindent For regular functions we can give the following natural definition of derivative.
\begin{defn}
Let $f: B(0,R) \rightarrow \HH$ be a regular function. The {\em slice derivative} (or {\em Cullen derivative}) of $f$ at $q=x+yI$ is defined as
\begin{equation*}
\partial_C f(x+yI)= \frac{\partial}{\partial x} f(x+yI).
\end{equation*}
\end{defn}
\noindent We remark that this definition is well posed because it is applied only to regular functions. Moreover,
since the operators $\partial_C$ and $\overline{\partial_I}$ do commute, the slice derivative of a regular function is still regular. Hence, we can iterate the differentiation obtaining (see for instance  \cite{GSAdvances}),
\[\partial_C^n f= \frac{\partial^n}{\partial x^n}f \quad \text{for any} \quad n \in \mathbb{N}.\]
 
\noindent As stated in \cite{GSAdvances}, a quaternionic power series $\sum_{n\geq 0}q^na_n$ with $\{a_n\}_{n\in \mathbb{N}}\subset \HH$ defines a regular function in its domain of convergence, which is a ball $B(0,R)$ with $R$ equal to the radius of convergence of the power series. Moreover, in \cite{GSAdvances}, it is also proved that
\begin{teo}
A function $f$ is regular on $B=B(0,R)$ if and only if $f$ has a power series expansion
$$f(q)=\sum_{n \geq 0}q^na_n \quad\text{with} \quad a_n=\frac{1}{n!}\frac{\partial^n f}{\partial x^n}(0).$$
\end{teo}  
\noindent A fundamental result in the theory of regular functions, that relates slice regularity and classical holomorphy, is the following, \cite{GSAdvances}:
\begin{lem}[Splitting Lemma]\label{split}
If $f$ is a regular function on $B=B(0,R)$, then for every $I \in \mathbb{S}$ and for every $J \in \mathbb{S}$, $J$ orthogonal to $I$, there exist two holomorphic functions $F,G:B_I \rightarrow L_I$, such that for every $z=x+yI \in B_I$, it holds
$$f_I(z)=F(z)+G(z)J.$$
\end{lem}
\noindent The following version of the Identity Principle is one of the first consequences, (as shown in  \cite{GSAdvances}):
\begin{teo}[Identity Principle]\label{Id}
Let $f$ be a regular function on $B=B(0,R)$. Denote by $Z_f$ the zero set of $f$, $Z_f=\{ q \in B | \, f(q)=0 \}$. If there exists $I \in \mathbb{S}$ such that $B_I \cap  Z_f$ has an accumulation point in $B_I$, then
$f$ vanishes identically on $B$.
\end{teo}

\noindent Another useful result is the following (see \cite{kernel, ext})
\begin{teo}[Representation Formula]\label{RF}
Let $f$ be a regular function on  $B=B(0,R)$ and let $J\in \mathbb{S}$. Then, for all $x+yI \in B$, the following equality holds
\begin{equation*}
f(x+yI)=\frac{1}{2}\big[ f(x+yJ)+f(x-yJ)\big]+I\frac{1}{2}\big[J\big[f(x-yJ)-f(x+yJ) \big]\big].
\end{equation*}
In particular for each sphere of the form $x+y\s$ contained in $B$, there exist $b, c \in \HH$ such that $f(x+yI)=b+Ic$\;  for all $I\in \s$ . 
\end{teo}

\noindent Thanks to this result, it is possible to recover the values of a function on more general domains than open balls centred at the origin, from its values on a single slice $L_I$. This yields an extension theorem (see \cite{kernel, ext}) that in the special case of functions that are regular on $B(0,R)$ can be obtained by means of their power series expansion.
\begin{oss}
If $f_I$ is a holomorphic function on a disc $B_I=B(0,R)\cap L_I$ and its power series expansion is
\[ f_I(z)=\sum_{n=0}^{\infty}z^na_n, \quad \text{ with $\{a_n\}_{n\in \mathbb{N}}\subset \HH$}, \]
then the unique regular extension of $f_I$ to the whole ball $B(0,R)$ is the function defined as
\[ \ext (f_I) (q) = \sum_{n=0}^{\infty}q^na_n .\]
The uniqueness is guaranteed by the Identity Principle \ref{Id}.
\end{oss}

\noindent If we multiply pointwise two regular functions in general we will not obtain a regular function. To guarantee the regularity of the product we need to introduce a new multiplication operation, the $*$-product. On open balls centred at the origin we can define the $*$-product of two regular functions by means of their power series expansions, see \cite{zeri}. 

\begin{defn}
Let $f,g: B=B(0,R) \rightarrow \HH$ be regular functions and let
\[f(q)=\sum_{n=0}^{\infty}q^na_n, \quad g(q)=\sum_{n=0}^{\infty}q^nb_n\]
be their  series expansions.
The {\em regular product} (or {\em $*$-product}) of $f$ and $g$ is the function defined as
\[
f*g(q)=\sum_{n=0}^{\infty}q^n \bigg( \sum_{k=0}^{n}a_kb_{n-k}\bigg),\]
regular on $B$.
\end{defn}

\noindent Notice that the $*$-product is associative but generally it is not commutative. Its connection with the usual pointwise product is stated by the following result.
\begin{pro}\label{trasf}
Let $f(q)$ and $g(q)$ be regular functions on $B=B(0,R)$. Then, for all $q\in B$,
\begin{equation}\label{prodstar}
f*g(q)= \left\{ \begin{array}{ll}
 f(q)g(f(q)^{-1}qf(q)) & \text{if} \quad f(q)\neq 0\\
0 & \text{if} \quad f(q)=0
\end{array}
\right.
\end{equation}
\end{pro}

\noindent We remark that if $q=x+yI$ (and if $f(q)\neq 0$), then $f(q)^{-1}qf(q)$ has the same modulus and same real part as $q$. Therefore $f(q)^{-1}qf(q)$ lies in the same $2$-sphere $x+y\mathbb{S}$ as $q$. We obtain then that a zero $x_0+y_0I$ of the function $g$ is not necessarily a zero of $f*g$, but an element on the same sphere $x_0+y_0\mathbb{S}$ does. In particular a real zero of $g$ is still a zero of $f*g$.
To present a characterization of the structure of the zero set of a regular function $f$ we need to introduce the following functions.

\begin{defn}\label{conj}
Let $f(q)=\sum_{n=0}^{\infty}q^na_n$ be a regular function on $B=B(0,R)$. We define the {\em regular conjugate} of $f$ as 
\[f^c(q)=\sum_{n=0}^{\infty}q^n\overline{a_n},\]
and the {\em symmetrization} of $f$ as
\[f^s(q)=f*f^c(q)=f^c*f(q)=\sum_{n=0}^{\infty}q^n\bigg(\sum_{k=0}^na_k\overline{a_{n-k}}\bigg).\]
\noindent Both $f^c$ and $f^s$ are regular functions on $B$.
\end{defn}
 
\begin{oss}\label{f^c}
Let $f(q)=\sum_{n=0}^{\infty}q^na_n$ be regular on $B=B(0,R)$ and let $I\in \s$. Consider the splitting of $f$ on $L_I$, $f_I(z)=F(z)+G(z)J$ with $J\in\s$, $J$ orthogonal to $I$ and $F,G$ holomorphic functions on $L_I$. 
In terms of power series, if $F(z)=\sum_{n=0}^{\infty}z^n\alpha_n$ and $G(z)=\sum_{n=0}^{\infty}z^n\beta_n$ we have
\begin{equation*}
\begin{aligned} 
f_I(z)=\sum_{n=0}^{\infty}z^na_n=\sum_{n=0}^{\infty}z^n\alpha_n+\sum_{n=0}^{\infty}z^n\beta_nJ
=\sum_{n=0}^{\infty}z^n(\alpha_n+\beta_nJ).
\end{aligned}
\end{equation*}
Hence the regular conjugate has splitting
\begin{equation*}
\begin{aligned} 
f^c_I(z)=\sum_{n=0}^{\infty}z^n(\overline{\alpha_n+\beta_nJ})=\sum_{n=0}^{\infty}z^n\overline{\alpha_n}-\sum_{n=0}^{\infty}z^n\beta_nJ.
\end{aligned}
\end{equation*}
That is 
\begin{equation*}
f_I^c(z)=\overline{F(\overline{z})}-G(z)J.
\end{equation*}
\end{oss}

\noindent The function $f^s$ is slice preserving (see \cite{ext}), i.e. $f^s(L_I)\subset L_I$ for every $I \in \s$. 
Thanks to this property it is possible to prove (see for instance \cite{zeri}) that the zero set of a regular function that does not vanishes identically, consists of isolated points or isolated $2$-spheres of the form $x +y \mathbb{S}$ with $x,y \in \mathbb{R}$, $y \neq 0$. 

\noindent 
A calculation shows that the slice derivative satisfies the Leibniz rule with respect to the $*$-product. 
\begin{pro}[Leibniz rule]\label{Leib}  
Let $f$ and $g$ be regular functions on $B=B(0,R)$. Then
$$\partial_C \big (f * g)(q)=\partial_C f * g(q) + f* \partial_C g (q)$$
for every $q \in B$.
\end{pro}

\noindent A basic result in analogy with the complex case, is the following (see \cite{open}).
\begin{teo}[Maximum Modulus Principle]\label{PMM}
Let $f: B=B(0,R) \rightarrow \HH$ be a regular function. If $|f|$ has a local maximum in $B$, then $f$ is constant in $B$. 
\end{teo}

\noindent In \cite{weierstrass} it is proved that we can estimate the maximum modulus of a function with its maximum modulus on each slice.
\begin{pro}\label{rep}
Let $f$ be a regular function on $B=B(0, R)$. If there exist an imaginary unit $I \in \s$ and a real number $m\in (0,+\infty)$ such that \[f_I(B_I)\subset B(0,m),\] 
then
\[ f(B) \subset B(0, 2m).\] 
In particular, if we set $m=\sup_{z\in B_I}|f(z)|$, then 
\begin{equation}\label{dis1}
\sup_{q \in B}|f(q)|\leq 2 m.
\end{equation}
\end{pro}

\section{Uniform norm and regular conjugation}\label{3}

This section is devoted to prove that the uniform norm on an open ball centred at the origin is the same for a regular function and for its regular conjugate.

\begin{pro}\label{cS}
Let $c$ be in $\HH$. Then the sets $\{cI | I \in \s \}$ and $\{Ic | I \in \s \}$ do coincide.
\end{pro}
\begin{proof} %\smartqed
Let $c=a+bJ$ with $a,b \in \rr$ and $J \in \s$. Let us fix $I\in\s$. We want to find an element $L$ of $\s$ such that $cI=Lc$, that is such that 
\begin{equation}\label{bS}
aI+bJI=aL+bLJ.
\end{equation}
Let us denote by $\langle \ ,\ \rangle$ the usual scalar product and by $\times$ the vector product. Recalling that for all imaginary units $I,J \in \s$ the following multiplication rule holds, see \cite{GSAdvances},
\[
IJ=-\langle I,J\rangle +I\times J,
\]
we can write equation \eqref{bS} as
\begin{equation*}
aI+b(-\langle J,I \rangle +J \times I)=aL+b(-\langle L,J \rangle + L \times J).
\end{equation*}
If we complete $J$ to a orthonormal basis  $1,J,K, JK$, of $\HH$ over $\rr$, then we can decompose
\[
I=i_1J+i_2K+i_3JK \quad \text{and} \quad L=l_1J+l_2K+l_3JK,
\]
obtaining that
\[
J\times I= -i_3K+i_2JK \quad \text{and} \quad L\times J=l_3K-l_2JK.
\]
Hence we need an imaginary unit $L$ such that
\begin{equation*}
\begin{aligned}
& a(i_1J+i_2K+i_3JK)+b(-i_1 +-i_3K+i_2JK)\\
&=a(l_1J+l_2K+l_3JK)+b(-l_1 + l_3K-l_2JK).
\end{aligned}
\end{equation*}
Considering the different components along $1,J,K,JK$, we get that $L$ has to satisfy the following linear system
\begin{equation}\label{sistema}     
\left\{
\begin{array}{rl}
 &bi_1=bl_1\\
 &ai_1=al_1\\
  &ai_2-bi_3=al_2+bl_3\\
  &bi_2+ai_3=-bl_2+al_3
\end{array} \right.
\end{equation}
that has a unique solution $(l_1,l_2, l_3)$ determining $L$.
%The determinant of the associated matrix is $0$ if and only if $a$ or $b$ are equal to zero. These special conditions correspond respectively to the case where $c$ is purely imaginary and where $c$ is real. In both cases the statement of the proposition is trivially true. For every other values of $a$ and $b$, the linear system \eqref{sistema} has a unique solution that gives us the imaginary unit $L$ that we were looking for.
%\qed 
\end{proof}

\begin{pro}\label{norma}
Let $f$ be a regular function on $B=B(0,R)$. For any sphere of the form $x+y\s$ contained in $B$ the following equalities hold: 
\[\inf_{I \in\s}|f(x+yI)|=\inf_{I \in\s}|f^c(x+yI)| \quad \text{and} \quad \sup_{I \in \s}|f(x+yI)|=\sup_{I \in \s}|f^c(x+yI)|
\]
\end{pro}
\begin{proof} %\smartqed 
Let $q=x+yI$ be an element of $B$. Theorem \ref{RF} yields that $f$ is affine on the sphere  $x+y\s$ and there exist $b, c\in \HH$ such that  $f(x+yI)=b+Ic$ for all $I\in \s$. We want to compare now the value of $f$ with the one attained by $f^c$ by means of their power series expansions.
%\[
%f(x+yI)=\frac{1}{2}\big[ f(x+yJ)+f(x-yJ)\big]+I\frac{1}{2}\big[J\big[f(x-yJ)-f(x+yJ) \big]\big],
%\]
If $f$ has power series expansion $f(q)=\sum_{n\geq0}q^na_n$ and we set $w=x+yJ$, then by Theorem \ref{RF} we get
\begin{equation}
\begin{aligned}
f(q)=f(x+yI)&=\frac{1}{2}\bigg( \sum_{n\geq0}w^na_n+\sum_{n\geq0}\overline{w^n}a_n\bigg)+\frac{IJ}{2}\bigg(\sum_{n\geq0}\overline{w^n}a_n-\sum_{n\geq0}w^na_n \bigg)\\
&=\frac{1}{2} \sum_{n\geq0}(w^n+\overline{w^n})a_n+\frac{IJ}{2}\sum_{n\geq0}(\overline{w^n}-w^n)a_n \\
&=\frac{1}{2} \sum_{n\geq0}2\RRe(w^n)a_n+\frac{IJ}{2}\sum_{n\geq0}-2\IIm(w^n)a_n \\
&=\sum_{n\geq0}\RRe(w^n)a_n+I\sum_{n\geq0}|\IIm(w^n)|a_n. \\
\end{aligned}
\end{equation}
Hence the constants $b$ and $c$ are  
\[b=\sum_{n\geq0}\RRe(w^n)a_n \quad \text{and} \quad c=\sum_{n\geq0}|\IIm(w^n)|a_n.\]
Since the power series expansion of $f^c$ is $f^c(q)=\sum_{n\geq0}q^n\overline{a_n}$, we obtain that for all $I\in \s$
\[
f^c(x+yI)=\sum_{n\geq0}\RRe(w^n)\overline{a_n}+I\sum_{n\geq0}|\IIm(w^n)|\overline{a_n}.
\]
Notice that $\RRe(w^n)$ and $|\IIm(w^n)|\in \rr$ for all $n\geq0$, then, in terms of $b$ and $c$, we can write
\[
f^c(x+yI)=\overline{b}+I\overline{c}
\]
for all $I$ in $\s$.
%Clearly
%\[
%\inf_{I\in \s}|f^c(x+yI)|\leq |f^c(x+yI)| \leq \sup_{I\in \s}|f^c(x+yI)|
%\]
%Hence
%\[
%\inf_{I\in \s}|f^c(x+yI)| \inf_{I\in \s}|\overline{a}+I\overline{b}|\leq |f^c(x+yI)| \leq \sup_{I\in \s}|\overline{a}+I\overline{b}|.
%\]
Hence
\[\sup_{I\in \s}|f^c(x+yI)|=\sup_{I\in \s}|\overline{b}+I\overline{c}|=\sup_{I\in \s}|\overline{\overline{b}+I\overline{c}}|=\sup_{I\in \s}|b-cI|=\sup_{I\in \s}|b+cI|.
\]
By Proposition \ref{cS} we obtain that
\[
\sup_{I\in \s}|f^c(x+yI)|=\sup_{I\in \s}|b+cI|=\sup_{I\in \s}|b+Ic|=\sup_{I\in\s}|f(x+yI)|.
\] 
Exactly the same arguments hold for the infimum, so we can conclude also that
\[\inf_{I\in\s} |f^c(x+yI)=\inf_{I\in\s}|f(x+yI)|.\]
%\qed 
\end{proof}

\begin{coro}\label{inf}
Let $f$ be a regular function on $B=B(0,R)$. Then
\[
\sup_{q \in B}|f(q)|=\sup_{q \in B}|f^c(q)| \quad \text{and} \quad \inf_{q\in B}|f(q)|=\inf_{q\in B}|f^c(q)|.
\]
%If $||\cdot||_{B}$ denotes the uniform norm on $B$, $||f||_{B}=\sup_{q \in B}|f(q)|$, we have  
%\[
%||f||_{B}=||f^c||_{B}.
%\]
\end{coro}
\begin{proof} %\smartqed
If $S$ is the subset of $\rr^2$ defined as
\[
S=\{(x,y)\in \rr^2 \ | \ y \geq 0, x^2+y^2 \leq R^2\},
\]
then we can cover the entire ball $B$ with spheres of the form $x+y\s$ as
\[B=\bigcup_{(x,y)\in S}x+y\s=\bigcup_{(x,y)\in S}\bigcup_{I\in\s}x+yI.\]
By Proposition \ref{norma}  we get
\begin{equation*}
\begin{aligned}
\sup_{q \in B}|f(q)|&=\sup_{(x,y)\in S}\sup_{I\in\s}|f(x+yI)|\\
&=\sup_{(x,y)\in S}\sup_{I\in\s}|f^c(x+yI)|=\sup_{q\in B}|f^c(q)|
\end{aligned}
\end{equation*}
and the same holds for the infimum, hence we have also
\[
\inf_{q\in B}|f(q)|=\inf_{q\in B}|f^c(q)|.
\]
%\qed 
\end{proof}

\section{A norm for a mean value theorem}\label{4}

The main technical tool to prove an analog of the Bloch-Landau Theorem for regular functions is stated in terms of a new norm. This norm is equivalent to the uniform norm, on the space of functions that are regular on an open ball $B$ centred at the origin, $B=B(0,R)$. The motivation to introduce this norm relies upon one of its properties, stated at the end of the section, which is useful to prove a mean value theorem.

\noindent Let $f:B\rightarrow \HH$ be a regular function. Take $I,J \in \s$, $I$ orthogonal to $J$ and, according to the Splitting Lemma \ref{split}, let $F,G:B_I\rightarrow L_I$ be the holomorphic functions such that the restriction of $f$ to $B_I=B\cap L_I$ is
$$f_I(z)=F(z)+G(z)J.$$
Let $\Omega$ be a subset of the ball $B$, and let $||\cdot||_{\Omega}$ denote the uniform norm on  $\Omega \subseteq B$, $$||\cdot||_{\Omega}=\sup_{\Omega}|\cdot|$$ 
%and let $||\cdot||_{B_I}$ be the uniform norm on the disc $B_I$, $$||\cdot||_{B_I}=\sup_{B_I}|\cdot|$$ 
For any $I\in \s$, we will indicate with $|| \cdot ||_I$ the function 
\begin{equation*}
||\cdot ||_I : \{f: B \rightarrow \HH\  |\ f \ \text{is regular} \ \} \rightarrow [0, +\infty) 
\end{equation*}
defined by
\[||f||_I^2=||F||_{B_I}^2+||G||_{B_I}^2.\]

\begin{oss}
For all $I \in \s$, the function $||\cdot||_I$ does not depend on $J$; in fact, if we choose another imaginary unit $K\in \s$, orthogonal to $I$, then the splitting of $f$ on $L_I$ is
\[f_I(z)=\tilde{F}(z)+\tilde{G}(z)K,\]
where $\tilde{F}(z)=F(z)$, because $I$ and $K$ are orthogonal, and hence $\tilde{G}(z)=G(z)JK^{-1}$. Then 
$$|\tilde{G}(z)|=|G(z)JK^{-1}|=|G(z)|$$
for all $z$ in $B_I$, and hence $||f||_I$ does not change. 
\end{oss}

\noindent Consider now the function 
$$||\cdot||: \{f:B\rightarrow \HH \ | \ f \ \text{ is regular} \}\rightarrow [0,+\infty)$$
defined by 
$$ ||f||=\sup_{I \in \s} ||f||_I .$$

\begin{pro}
The function $||\cdot||$ is a norm on the real vector space  $\mathcal{B}=\{f:B\rightarrow \HH\  | \ f \ \text{ is regular} \}$. 
%Moreover, for all $c\in \HH$ and $f\in \mathcal{B}$
%\begin{equation}\label{quasiomogenea}
%|c|||f||\leq ||fc|| \leq \sqrt{2}|c|||f||
%\end{equation}
\end{pro}
\begin{proof} %\smartqed
Let $f\in \mathcal{B}$, $I\in \s$, and take $J\in \s$ orthogonal to $I$. Let $F$, and $G$ be the holomorphic functions on $L_I$, such that $f_I(z)=F(z)+G(z)J$ for all $z \in B_I$. Then: 
\begin{itemize}
\item $||f||=0$ if and only if, for all $I \in \s$, 
$$0=||f||_I^2= ||F||_{B_I}^2+||G||_{B_I}^2,$$ 
and hence if and only if $F=G=0$. By the Identity Principle \ref{Id} we can conclude that $||f||=0$ if and only if $f=0.$
\item Let $c\in \rr$. Then the splitting of $fc$ on $L_I$ is $(fc)_I(z)= F(z)c+G(z)cJ.$ Hence, using the homogeneity of the uniform norm, we have
\begin{equation*}
\begin{aligned}
||fc||^2&=\sup_{I\in \s}||fc||^2_I\\
&=\sup_{I\in \s}\big(||F||^2_{B_I}|c|^2+||G||^2_{B_I}|c|^2\big)\\
&=\sup_{I \in \s}\big(||F||^2_{B_I}+||G||^2_{B_I}\big)|c|^2=|c|^2||f||^2.
\end{aligned}
\end{equation*}
\item If $F_j,G_j$ are the splitting functions of regular functions $f_j$ with respect to $I$ and $J$ for $j=1,2$, then 
\begin{equation}\label{triangle1}
\begin{aligned}
&(||f_1||_I+||f_2||_I)^2\\
& =\Big( \sqrt{||F_1||^2_{B_I}+||G_1||^2_{B_I}} + \sqrt{||F_2||^2_{B_I}+||G_2||^2_{B_I}} \Big)^2\\
&= ||F_1||^2_{B_I}+||G_1||^2_{B_I} + ||F_2||^2_{B_I}+||G_2||^2_{B_I} \\
& \hspace{1cm} + 2 \sqrt{(||F_1||_{B_I}^2 +||G_1||_{B_I}^2)(||F_2||_{B_I}^2+||G_2||_{B_I}^2)},\\
\end{aligned}
\end{equation}
and
\begin{equation}\label{triangle2}
\begin{aligned}
&||f_1+f_2||_I^2=||F_1+F_2||^2_{B_I}+||G_1+G_2||^2_{B_I}\\
&=||(F_1+F_2)^2||_{B_I}+||(G_1+G_2)^2||_{B_I}\\
&=||F_1^2+F_2^2+2F_1F_2||_{B_I}+||G_1^2+G_2^2+2G_1G_2||_{B_I}\\
& \leq ||F_1^2||_{B_I}+||F_2^2||_{B_I}+2||F_1||_{B_I}||F_2||_{B_I}\\
& \hspace{2cm} +||G_1^2||_{B_I}+||G_2^2||_{B_I}+2||G_1||_{B_I}||G_2||_{B_I}.  
\end{aligned}
\end{equation}
The last quantity in equation \eqref{triangle2} is less or equal than $(||f_1||_I+||f_2||_I)^2$  if and only if
\begin{equation*}
\begin{aligned}
&||F_1||_{B_I}||F_2||_{B_I}+||G_1||_{B_I}||G_2||_{B_I}\\
&\leq \sqrt{(||F_1||_{B_I}^2 +||G_1||_{B_I}^2)(||F_2||_{B_I}^2+||G_2||_{B_I}^2)},
\end{aligned}
\end{equation*}
that is, if and only if
\begin{equation*}
\begin{aligned}
&\langle(||F_1||_{B_I},||G_1||_{B_I}),(||F_2||_{B_I},||G_2||_{B_I})\rangle\\
&\leq \sqrt{||F_1||_{B_I}^2 +||G_1||_{B_I}^2}\sqrt{||F_2||_{B_I}^2+||G_2||_{B_I}^2},
\end{aligned}
\end{equation*}
that holds thanks to Cauchy-Schwarz inequality for the scalar product on $\rr^2$. 
\end{itemize}
\noindent Therefore the function $|| \cdot||$ is a norm.

\end{proof}

\noindent Let us now show that the norms $||\cdot||$ and $||\cdot||_{B}$, defined on $\mathcal{B}$, are equivalent.
\begin{pro}\label{equivalenza}
Let $f:B=B(0,R) \rightarrow \HH$ be a regular function. Then
$$\frac{\sqrt{2}}{2}||f||\leq||f||_{B}\leq ||f||.$$ 
\end{pro}

\begin{proof} %\smartqed
Let $I,J \in \s$, $I$ orthogonal to $J$, and let $F,G$ be holomorphic functions on $L_I$, such that $f_I(z)=F(z)+G(z)J$ for all $z \in B_I$. Then we have

\begin{equation*}
\begin{aligned}
||f||_{B}^2&=\sup_{q\in B}|f(q)|^2=\sup_{I\in\s} \sup_{z\in B_I}|f_I(z)|^2\\
&=\sup_{I\in\s} \sup_{z\in B_I}|F(z)+G(z)J|^2=\sup_{I\in\s} \sup_{z\in B_I}\big(|F(z)|^2+|G(z)|^2\big)\\
&\leq \sup_{I\in\s} \big( \sup_{z\in B_I}|F(z)|^2 + \sup_{z\in B_I}|G(z)|^2 \big)\\
&=\sup_{I\in\s} \big( ||F||_{B_I}^2 + ||G||^2_{B_I} \big)=||f||^2.
\end{aligned}
\end{equation*}

\noindent Conversely
\begin{equation*}
\begin{aligned}
||f||^2&= \sup_{I\in\s}||f||_I^2=\sup_{I\in\s} \big( ||F||_{B_I}^2 + ||G||^2_{B_I} \big)\\
&\leq \sup_{I\in \s }\big( ||f_I||_{B_I}^2 + ||f_I||^2_{B_I} \big)=2 ||f||^2_{B}.
\end{aligned}
\end{equation*}
%\qed 
\end{proof}

\noindent As we anticipated at the beginning of this section, in terms of the norm $|| \cdot ||$ we can state (and prove) a mean value theorem.
\begin{teo}\label{media}
Let $f$ be a regular function on $B=B(0,R)$ such that $f(0)=0$. Then
$$|q^{-1}f(q)|\leq ||\partial_Cf||$$
for all $q \in B$.
\end{teo}

\begin{proof} %\smartqed
Let $I\in \s$ and take $J \in \s $ orthogonal to $I$. By the Splitting Lemma \ref{split} there exist  $F,G:B_I\rightarrow L_I$ holomorphic functions, such that $f_I(z)=F(z)+G(z)J$ for all $z\in B_I$. 
By the Fundamental Theorem of Calculus for holomorphic functions (see for instance Theorem 3.2.1 in book \cite{Lang}),
we get that
$$|z^{-1}F(z)|\leq||F'||_{B_I}  \quad \text{and} \quad |z^{-1}G(z)|\leq ||G'||_{B_I}$$
for all $z\in B_I$. Hence
\begin{equation}\label{dis2}
\begin{aligned}
|z^{-1}f(z)|^2&=|z^{-1}F(z)+z^{-1}G(z)J|^2=|z^{-1}F(z)|^2+|z^{-1}G(z)|^2\\
&\leq ||F'||^2_{B_I}+||G'||^2_{B_I}=||\partial_Cf||^2_I\\
&\leq\sup_{I\in\s}||\partial_Cf||^2_I=||\partial_Cf||^2.
\end{aligned}
\end{equation}
Since $||\partial_Cf||$ does not depend on $I$, we have that inequality \eqref{dis2} holds for every $q\in B$.
%\qed 
\end{proof}

\begin{oss}\label{media2}
As a consequence of Theorem \ref{media} and of the Maximum Modulus Principle \ref{PMM}, we get also that if $f$ is regular on $B(0,R)$, then for all $r\in (0,R)$
\begin{equation}
\sup_{q\in B(0,r)}|f(q)|\leq r||\partial_Cf||.
\end{equation}
\end{oss}

\noindent We conclude this section showing that, as the uniform norm, the norm $||\cdot||$ satisfies the following
\begin{pro}
Let $f$ be a regular function on $B=B(0,R)$. Then $||f||=||f^c||$.
\end{pro}
\begin{proof} %\smartqed
Let $I \in \s$. By Remark \ref{f^c}, if the splitting of $f_I$ is $f_I(z)=F(z)+G(z)J$ for all $z\in B_I$, then the regular conjugate of $f$ splits as 
$$f_I^c(z)=\overline{F(\overline{z})}-G(z)J.$$
Since the (complex) conjugation is a bijection of $B_I$ and the modulus $|\overline{F(z)}|$ is equal to $|F(z)|$ for all $z\in B_I$, we get that
\begin{equation}\label{normaf^c}
||f^c||_I^2=||\overline{F(\overline{z})}||^2_{B_I}+||-G||^2_{B_I}=||F||^2_{B_I}+||G||^2_{B_I}=||f||_I^2.
\end{equation}
Since equality \eqref{normaf^c} holds for every $I$ in $\s$, we can conclude.
%\qed 
\end{proof}

\section{The Bloch-Landau  type Theorem}
\noindent 
%First we need to prove a technical result, which is a fundamental tool to prove the Bloch-Landau Theorem for regular functions.
For $\varrho >0$, let us denote by $\mathcal{O}(\varrho)$ the open set 
\[
\mathcal{O}(\varrho)=\{q\in \HH \, | \, |q|^3 < \varrho |\RRe(q)|^2  \}.
\]
Notice that the intersection of $\mathcal{O}(\varrho)$ with a slice $L_I=\{x+yI \ | x,y \in \rr \}$ is the interior of an eight shaped curve with equation
\[
(x^2+y^2)^{\frac{3}{2}}=\varrho x^2.
\]
%where $x+yI$ is an element of $L_I$.
We remark that this curve always contains two discs with positive radius depending just on $\varrho$. For example, we can take the discs centred in $(\frac{\varrho}{2},0)$ and in $(-\frac{\varrho}{2},0)$, of radius $\frac{37}{256}\varrho^2$.
Therefore, the open set $\mathcal{O}(\varrho)$ always contains two open balls of radius at least $\frac{37}{256}\varrho^2$. 
\begin{lem}\label{lemma}
Let $f:B=B(0,R)\rightarrow \HH$ be a (non constant) regular function such that $f(0)=0$ and that $\de f(0) \in \rr$.
Then the image of $B$ under $f$ contains an open set of the form $\mathcal O(\varrho)$ 
%whose intersection with every slice $L_I$ is the interior of an eight shaped curve $\mathcal C$ with equation 
%\[
%(x^2+y^2)^{\frac{3}{2}}=\varrho x^2 
%\]
where 
\[
\varrho=\frac{R|\partial_C f(0)|^2}{4||\partial_C f||}.
\]
\end{lem}

\begin{proof} %\smartqed 
Since $f(0)=0$, if $\partial_Cf(0)=0$ there is nothing to prove. Suppose then that $\partial_Cf(0)\neq 0$.
Consider a point $c$ outside the image of $B$ under $f$, then $c\neq 0$. We want to show that $c$ does not belong to $\mathcal O(\varrho)$.

\noindent For all $q\in B$, define $g(q)$  to be
\[
g(q)=(1-f(q)c^{-1})^s.
\]
The function $g$ is regular on $B$ and we can estimate its modulus in the following manner:
let $\tau(q)$ be the transformation defined by 
$$\tau(q)=(1-f(q)c^{-1})^{-1}q(1-f(q)c^{-1}).$$ 
Then, according to Proposition \ref{trasf} and recalling that $|\tau(q)|=|q|$ for all $q$, we can write
\begin{equation*}
\begin{aligned}
|g(q)|&=|(1-f(q)c^{-1})*(1-f(q)c^{-1})^c|=|(1-f(q)c^{-1})||(1-f(\tau(q))c^{-1})^c|\\
&\leq \sup_{|q|< R}|(1-f(q)c^{-1})|\sup_{|q|< R}|(1-f(q)c^{-1})^c|.
\end{aligned}
\end{equation*}
Hence, using Proposition \ref{norma}, 
\begin{equation*}
|g(q)|\leq \left(\sup_{|q|< R}|(1-f(q)c^{-1})|\right)^2,
\end{equation*}
that is
\[
|g(q)|^{\frac{1}{2}}\leq \sup_{|q|< R}|(1-f(q)c^{-1})|.
\]
By the properties of the uniform norm and by Remark \ref{media2} we get then
\begin{equation}\label{|g|1/2}
|g(q)|^{\frac{1}{2}}\leq 1 + \sup_{|q|< R}|f(q)c^{-1}|\leq 1 +|c|^{-1} \sup_{|q|< R}|f(q)|\leq1 + |c|^{-1}||\de f|| R .
\end{equation}

\noindent The next step is to estimate from below the quantity $|g(q)|^{\frac{1}{2}}$. Notice that $g$ is slice preserving, since it is the symmetrization of a regular function. Moreover $g(q)$ is never zero. For all $I$ in  $\s$ the map $z \mapsto z^4$ from  $L_I\setminus \{0\} \rightarrow L_I \setminus \{0\}$ is a covering map. Since $B_I$ is simply connected, we can lift the function $g$ obtaining a holomorphic function $\Psi_I:B_I\rightarrow L_I \setminus \{0\}$ such that 
$$\Psi_I(z)^4=(1-f(z)c^{-1})^s$$
and $\Psi_I(0)=1$ (since $g(0)=1$).
Let $\Psi$ be the (unique) extension to $B$ of $\Psi_I$. 
%From the Identity Principle \ref{Id} it follows that $\Psi$ is the unique extension of $\Psi^J$ for all $J\in \s$.   
Now $\Psi(0)=1$,
\begin{equation*}
\Psi(q)^4=(1-f(q)c^{-1})^s=g(q)
\end{equation*}
for all $q \in B$ and in particular 
\begin{equation}\label{radice2}
|\Psi(q)|^2=|g(q)|^{\frac{1}{2}} \quad \text{for all} \ \ q \in B.
\end{equation}
We want to use the power series expansion of $\Psi$ to find a lower bound of $|g|^{\frac{1}{2}}$. In particular we need to compute its slice derivative. Using the Leibniz rule \ref{Leib} we can calculate 
\begin{equation*}
\begin{aligned}
\partial_C g(q)&=\partial_C [(1-f(q)c^{-1})^s] =\partial_C [(1-f(q)c^{-1})*(1-f(q)c^{-1})^c]\\
&=-\partial_Cf(q)c^{-1}*(1-(f^c(q)c^{-1})^{c})- (1-f(q)c^{-1})* \partial_C (f(q)c^{-1})^{c}.
\end{aligned}
\end{equation*}
Since $q=0$ is a real zero of $f$ (and hence of $f^c$), and since the operators of slice differentiation and regular conjugation do commute, if the power series expansion of $f$ is $f(q)=\sum_{n=0}^{\infty}q^na_n$, we obtain
that 
\[
(\de f (q) c^{-1})^c=\sum_{n=1}^{\infty}q^{n-1}\overline{na_nc^{-1}}
\]  
and hence that
$$\partial_C g(0)=-\partial_C f(0)c^{-1}-\overline{\partial_C f(0)c^{-1}}= - 2 \RRe(\partial_C f(0)c^{-1}).$$
Moreover, since $g$ (and $\Psi$) is slice preserving, $g(0)=1$, and $\de f(0)$ is real, we have
\begin{equation}\label{|psi|2}
\partial_C \Psi(0)=\frac{1}{4}g(0)^{\frac{1}{4}-1}\partial_C g(0)= - \frac{1}{2}\RRe(\partial_C f(0)c^{-1})=- \frac{1}{2}\partial_C f(0)\RRe(c^{-1}).
\end{equation}

%da qui come nel caso complesso
\noindent Let us set
\[
M:=1+ |c|^{-1}||\de f|| R.
\]
Fix $r\in(0,R)$ and $I\in \s$. Let $q \in \partial B_I(0,r)$, $q=re^{I\theta}$ for some $\theta \in [0,2\pi)$. By equations  \eqref{|g|1/2} and \eqref{radice2} we obtain that for every $\theta \in [0,2\pi)$
\begin{equation}\label{M}
M \geq |\Psi(re^{I\theta})|^2
\end{equation}
%If we integrate both sides of inequality \eqref{M}
%we get
%\begin{equation}
%%\begin{aligned}
%\frac{1}{2\pi}\int_0^{2\pi}M d\theta \geq \frac{1}{2\pi}\int_0^{2\pi}|\Psi(re^{I\theta})|^2d\theta.
%%\end{aligned}
%\end{equation}

\noindent Using the series expansion of $\Psi$ we can write
\begin{equation}\label{|psi|}
\begin{aligned}
|\Psi(re^{I\theta})|^2&=\overline{\Psi(re^{I\theta})}\Psi(re^{I\theta})\\
&=\Bigg(\sum_{m=0}^{\infty}\frac{1}{m!}r^m\overline{\Psi^{(m)}(0)}e^{-Im\theta}\Bigg)\Bigg(\sum_{n=0}^{\infty}\frac{1}{n!}r^ne^{In\theta}\Psi^{(n)}(0)\Bigg)\\
&=\sum_{m,n=0}^{\infty}\frac{1}{m!n!}r^{m+n}\overline{\Psi^{(m)}(0)}e^{I(n-m)\theta}\Psi^{(n)}(0).\\
\end{aligned}
\end{equation}
If we integrate in $\theta$, then we get
\begin{equation*}
\frac{1}{2\pi}\int_0^{2\pi}|\Psi(re^{I\theta})|^2d\theta =\frac{1}{2\pi}\int_0^{2\pi}\sum_{m,n=0}^{\infty}\frac{1}{m!n!}r^{m+n}\overline{\Psi^{(m)}(0)}e^{I(n-m)\theta}\Psi^{(n)}(0) d\theta.
\end{equation*}

\noindent Thanks to the uniform convergence on compact sets of the series expansion, we can exchange the order of integration and summation. Then, since $\int_0^{2\pi}e^{Is\theta}d\theta=0$ if $s \in \mathbb{Z}, s\neq0$ and equals $2\pi$ otherwise, just the terms where $n=m$ survive and hence we get
\begin{equation*}
\frac{1}{2\pi}\int_0^{2\pi}|\Psi(re^{I\theta})|^2d\theta= \sum_{m=0}^{\infty}\frac{r^{2m}}{(m!)^2}\overline{\Psi^{(m)}(0)}\Psi^{(m)}(0)=\sum_{m=0}^{\infty}\frac{r^{2m}}{(m!)^2}|\Psi^{(m)}(0)|^2.
\end{equation*}
\noindent By inequality \eqref{M} and since $M$ is a constant, we obtain
\begin{equation*}
M=\frac{1}{2\pi}\int_0^{2\pi}M d\theta \geq \sum_{m=0}^{\infty}\frac{r^{2m}}{(m!)^2}|\Psi^{(m)}(0)|^2.
\end{equation*}
Considering just the first two terms of the series expansion and using equation \eqref{|psi|2}, we get
\begin{equation*}
M\geq 1+r^2  |\partial_C\Psi(0)|^2= 1+ \frac{r^2|\partial_C f(0)|^2|\RRe(c^{-1})|^2}{4}=1+\frac{r^2|\partial_C f(0)|^2|\RRe(c)|^2}{4|c|^4}.
\end{equation*}
Recalling the expression of $M$ we have then
\begin{equation}\label{M2}
1+ |c|^{-1}||\de f|| R \geq1+\frac{r^2|\partial_C f(0)|^2|\RRe(c)|^2}{4|c|^4},
\end{equation}
that is
\[
|c|^{3}\geq \frac{r^2|\partial_C f(0)|^2|\RRe(c)|^2}{4||\de f|| R }
\]
for all $r \in (0,R)$. Hence, if we take the limit as $r$ approaches $R$, we obtain
that if a point $c$ is outside the image $f(B)$, then it satisfies the following inequality
\[
|c|^{3}\geq \frac{R|\partial_C f(0)|^2|\RRe(c)|^2}{4||\de f|| }.
\]
That is equivalent to say that the image of $B$ under $f$ contains the open set
\[
\mathcal{O}(\varrho)=\{q\in \HH \ | \ |q|^3 < \varrho | \RRe (q)|^2\}.
\]
where 
\[
\varrho=\frac{R|\partial_C f(0)|^2}{4||\de f|| }.
\]
%Let us now fix $I\in \s $ and consider the intersection $ \mathcal{S} \cap L_I$.
%For all $x + yI \in \mathcal{S} \cap L_I$, the following inquality holds
%\[
%(x^2+y^2)^{\frac{3}{2}}< \varrho x^2.
%\]
%This means that $\mathcal{S} \cap L_I= \mathcal{O} \cap L_I$ for all $I \in \s $ and hence we can conclude that $\mathcal{S}=\mathcal{O}$.
%\qed 
 \end{proof}

\noindent 
Let $w\in \HH$, and let $\tau_{w}$ be the translation $q \mapsto q+w$. The composition of a regular function with $\tau_w$ is not regular in general. For our purposes we need to define a new notion of composition in this special case. Notice that if we restrict both functions to the slice $L_I$ containing $w$, then we can compose them obtaining a holomorphic function.
\begin{defn}
Let $f$ be a regular function on $B=B(0,R)$ and let $w=x+yI \in B$. We define the {\em regular translation} $\tilde{f}_w$ of $f$ to be the unique regular extension of $f_I \circ (\tau_w)_I$,
\[\tilde{f}_w(q)= \ext( f_I \circ (\tau_w)_I)(q), \]
regular on $B(0,R-|w|)$.
\end{defn}
\noindent In the proof of the main result, it will be useful the following
\begin{pro}\label{convergenza}
Let $f$ be a regular function on $B=B(0,R)$ and let $w_n$ be a convergent sequence in $B$, such that $\lim_{n\to \infty}w_n=w \in B$. Set 
\[
m=\max \left\{ \{|w_n|,n\in \N\}\cup\{|w|\}\right\}
\]
Then 
$\tilde{f}_{w_n}$ converges to $\tilde{f}_w$ uniformly on compact subsets of $B(0,R-m)$.
%\[\lim_{n \to \infty}\tilde{f}_{w_n}(q)=\tilde{f}_w(q)\]
%for all $q \in B(0,R-m)$.
\end{pro}
\begin{proof} %\smartqed  First of all notice that the maximum $m$ exists because of the convergence of the sequence $\{w_n\}$.
Clearly the sequence $\tau_{w_n}$ converges (uniformly on compact sets) to $\tau_w$. Moreover, if $w_n \notin \mathbb{R}$ frequently, then (up to a subsequence) we define
\[I_n=\frac{\IIm(w_n)}{|\IIm(w_n)|}\quad \text{and} \quad \lim_{n\to \infty}I_n=I.\]
If, instead, there exists a natural number $n_0$ such that $w_n$ is real for all $n>n_0$ then we choose any $I\in \s$ and set $I_n=I$ for all $n>n_0$ in what follows.  In both cases, $I \in \s$ is such that $w\in L_I$.

\noindent 
%By Theorem \ref{RF} we get that for all $n \in \N$, if $x+y I_n \in B(0,R)$, 
%\[f_{I_n}(x+yI_n)=b+I_nc\]
%and hence
%\[
%\lim_{n\to \infty}f_{I_n}(x+yI_n)=b+Ic.
%\]
By Theorem \ref{RF} we can write for all $q=x+yJ \in B(0,R-m)$
\begin{equation}
\begin{aligned}
&\tilde{f}_{w_n}(x+yJ)\\
&= \frac{1}{2}\left[ \tilde{f}_{w_n}(x+yI_n)+\tilde{f}_{w_n}(x-yI_n)\right]+\frac{JI_n}{2}\left[\tilde{f}_{w_n}(x-yI_n)-\tilde{f}_{w_n}(x+yI_n) \right]\\
&=\frac{1}{2}\left[ f_{I_n}(x+yI_n+w_n)+f_{I_n}(x-yI_n+w_n)\right]\\
& \hskip 4cm +\frac{JI_n}{2}\left[f_{I_n}(x-yI_n+w_n)-f_{I_n}(x+yI_n+w_n) \right]\\
&=\frac{1}{2}\left[ f(x+yI_n+w_n)+f(x-yI_n+w_n)\right]\\
& \hskip 4cm +\frac{JI_n}{2}\left[f(x-yI_n+w_n)-f(x+yI_n+w_n) \right].\\
\end{aligned}
\end{equation}
Since $f$ is a continuous function, again by Theorem \ref{RF} 
we can conclude that (uniformly on compact sets) 
\[
\lim_{n\to \infty}\tilde{f}_{w_n}(q)=\tilde{f}_w(q)
\]
for all $q\in B(0,R-m)$.
%\qed  
\end{proof}

%{\color{red}
%\noindent Notice that the uniform norm of the regular translation $\tilde f_w$ of $f$ is controlled by the uniform norm of the regular function $f$ itself.
%\begin{oss}
%Let $f:B=B(0,R)\to \HH$ be a regular function, and let $w=u+vI\in B$. If $x+yJ$ is an element of $B(0,R-|w|)$, by Theorem \ref{RF}, we have
%\begin{equation}\label{f_w}
%\begin{aligned}
%&|\tilde{f}_w(x+yJ)|\leq\frac{1}{2}\left|\tilde{f}_{w}(x+yI)+\tilde{f}_{w}(x-yI)\right|+\left|\frac{JI}{2}\left[\tilde{f}_{w}(x-yI)-\tilde{f}_{w}(x+yI) \right]\right|\\
%&=\frac{1}{2}\left[ |f(x+yI+w)+f(x-yI+w)|+|f(x-yI+w)-f(x+yI+w)| \right]\\
%&\leq\frac{1}{2}\left[2|f(x+yI+w)|+2|f(x-yI+w)|\right]\\
%&=|f(x+yI+w)|+|f(x-yI+w)|\\
%&\leq 2|| f||_{B(w,R-|w|)}.
%\end{aligned}
%\end{equation}
%Since inequality \eqref{f_w} holds for all $x+yJ$ in $B(0,R-|w|)$, we obtain
%\[
%||\tilde{f}_w||_{B(0,R-|w|)}\leq 2|| f||_{B(w,R-|w|)}.
%\]
%\end{oss}
%}

\noindent In order to prove the Bloch-Landau type theorem we need a last step.
\begin{pro}\label{continua}
Let $f:B(0,R)\rightarrow \HH$ be a regular function. Then $M: [0,R) \rightarrow \rr$ \  defined as
\[
M(s)=\max_{|q|\leq s} |f(q)| 
\]
is a continuous function.
\end{pro}
\begin{proof} 
%\smartqed 
If $f$ is constant the statement is trivially true. Let us suppose then that $f$ is not constant. 
The Maximum Modulus Principle \ref{PMM} yields that the function $M(s)$ is increasing and hence for any sequence $s_n$ converging (from above or from below) to $s$ there exists the limit $\lim_{n\to\infty}M(s_n)$. To show that the limit is equal to $M(s)$, consider first the sequence $ \left\{ s+\frac{1}{n} \right\}_{n \in \N}$.
Since $ B (0,s+\frac{1}{n})$ is relatively compact, we can find $q_n$ such that $M(s+\frac{1}{n})=|f(q_n)|$ for all $n \in \mathbb{N}$, and, up to subsequences, we can suppose that $q_n$ converges to $q_0 \in \partial B (0,s)$. Therefore we have
\[
M(s)\leq \lim_{n\to \infty} M(s+\frac{1}{n})=\lim_{n \to \infty}|f(q_n)|=|f(q_0)|,
\]
where the last equality is due to the continuity of the function $|f(q)|$. Moreover, by definition of $M$ we have that $|f(q_0)|\leq M(s)$ and hence that
\[
M(s)= \lim_{n\to \infty} M(s+\frac{1}{n}).
\]
Now $q_0$ lies in the closure $\overline{B(0,s)}$, hence we can find a sequence whose term $p_n$ has modulus $|p_n|=s-\frac{1}{n}$ for all $n\in \mathbb{N}$, such that $\lim_{n\to \infty}p_n=q_0.$
Then
\[
M(s)\geq \lim_{n\to\infty} M(s-\frac{1}{n})\geq \lim_{n\to\infty}|f(p_n)|=|f(q_0)|=M(s).
\]
Therefore we can conclude that $M$ is continuous.
%\qed  
\end{proof}

\noindent Finally we have all the tools to prove the announced Bloch-Landau type theorem for regular functions.
Let $\B$ be the unit open ball of $\HH$, \  $\B=\{q \in \HH\  | \ |q|<1 \}.$

\begin{teo}[a Bloch-Landau type theorem]\label{BlochLandau}
Let $f:\B \rightarrow \HH$ \ be a regular function such that $f(0)=0$ and $\partial_Cf(0)=1$. Then there exists $u\in \B$ such that the image of the regular translation $\tilde{f}_u$ of $f$ contains an open set obtained by means of a rotation and a translation of $\mathcal{O}(\varrho)$, where the ``radius'' $\varrho$ is at least $\frac{1}{32\sqrt{2}}$.
\end{teo}

\begin{proof} %\smartqed 
Let us set $M(t)$ to be the function defined on $[0,1)$ by 
\[ 
M(t)=\max_{|q|\leq t}|\de f (q)|,
\]
fix $r$ in $(0,1)$, and consider the function 
\[
\mu(s)=sM (r-s),
\]
defined for $s \in [0,r]$. By Proposition \ref{continua}, $\mu$ is a continuous function, $\mu(0)=0$, $\mu(s)\geq 0$ for all $s \in [0,r]$, and $\mu(r)=r$.
Set
\[
R=\frac{1}{2}\min\{s\ |\ \mu(s)=r\},
\]
then $0 < 2R\leq r$. Let $w \in \partial B (0,r-2R)$ be such that $|\partial_Cf(w)|=M(r-2R)$, i.e. by definition of $R$, such that $|\partial_Cf(w)|=\frac{r}{2R}$.
Let us restrict our attention to the slice $L_I$ containing $w$. Consider the function $\varphi_{I}:B(0,2R)\cap L_{I}\rightarrow \HH$, defined by
\[
\varphi_{I}(z)=\left( f(z+w)-f(w) \right)\frac{\overline{ \de f (w)}}{|\de f (w)|}.
\]
The function $\varphi_I$ is holomorphic on $B(0,2R) \cap L_I$, because 
\[
|q+w|\leq |q|+|w|\leq 2R + (r-2R)=r.
\]
Let $\varphi$ be the (unique) regular extension to the entire ball $B(0,2R)$ of $\varphi_I$. Then $\varphi(0)=0$ and $\partial_C \varphi(0)=|\partial_C f(w)|=\frac{r}{2R}$, hence $\varphi$ satisfies the hypotheses of Lemma \ref{lemma}.
%Consider the following regular function 
%\[
%\psi(q)=\tilde{\varphi}(q)\frac{\overline{\partial_C\tilde{\varphi}(0)}}{|\partial_C \tilde{\varphi}(0)|}.
%\]
%We have that $\psi(0)=0$ and
%\begin{equation}\label{psi'}
%\partial_C \psi(0)=|\partial_C \tilde{\varphi}(0)|=\frac{r}{2R}.
%\end{equation}

\noindent For $z\in B(0,R)\cap L_I$ we have that
\[
|\partial_C \varphi_I(z)|=|\partial_C f(z+w)|\leq M (|z+w|) \leq M (r-R) =\frac{\mu(R)}{R}.
\]
Since $\mu$ is continuous, $\mu(0)=0$ and $\mu(r)=r$, then
\[
\frac{\mu(R)}{R}<\frac{r}{R},
\]
otherwise there would exists $s<2R$ such that $\mu(s)=r$, a contradiction with the definition of $R$.
Therefore
\[
\partial_C\varphi_I(B(0,R)\cap L_I) \subset B (0,\frac{r}{R}).
\]
Statement \eqref{dis1} of Proposition \ref{rep} implies then that
\[
\partial_C\varphi(B(0,R)) \subset B (0,\frac{2r}{R}).
\]
Considering the uniform norm we obtain 
\[
||\partial_C\varphi||_{B(0, R)}\leq \frac{2r}{R},
\]
 and hence, by Proposition \ref{equivalenza},  
\begin{equation}\label{ultima}
||\partial_C\varphi||\leq \frac{2\sqrt{2}r}{R}
\end{equation}
on $B(0,R)$. Lemma \ref{lemma} yields then that $\varphi(B(0,R))$ contains an open set $\mathcal{O}(\varrho)$ 
where 
\[
\varrho=\frac{R|\partial_C \varphi (0)|^2}{4||\partial_C \varphi ||}\geq \frac{R\left(\frac{r}{2R}\right)^2}{4\frac{2\sqrt{2}r}{R}} = \frac{r}{32 \sqrt{2}}.
\]
Recalling the definition of $\varphi$, we get then
\[
\tilde{f}_w(B(0,R))-f(w)\supset \mathcal{O}(\varrho(r))\frac{\partial_C\varphi(0)}{|\partial_C \varphi(0)|},
\]
that yields
\[
f(w)+\mathcal{O}(\varrho(r))\frac{\partial_C\varphi(0)}{|\partial_C\varphi(0)|}\subset \tilde{f}_w (B(0,R))
\]
%where the last inclusion is due to the fact that 
%\[
%R+|w|=R+r-2R=r-R \leq1.
%\]

\noindent Therefore for all $r<1$ there exist $R_r>0$ and $w_r$, with modulus $|w_r|=r-2R_r$, such that the image of $B(0,R_r)$ through $\tilde{f}_{w_r}$ contains the open set $f(w_r)+\mathcal{O}(\varrho(r))\frac{\partial_C\varphi(0)}{|\partial_C\varphi(0)|}$. 
When $r$ approaches 1, by compactness, we can find subsequences $\{R_n\}$ and $\{w_n\}$, converging respectively to $R_0>0$ and  to $w_0\in \mathbb{B}$ (in fact $R_0=0$ would imply that $\mu(0)=1$ which is not).
Thanks to Proposition \ref{convergenza} we have then that $\tilde{f}_{w_n}$ converges (uniformly on compact sets) to $\tilde{f}_{w_0}$, and hence we get that the image of $\tilde{f}_{w_0}$ contains the open set 
\[
\mathcal{O}(\varrho)\frac{\partial_C\varphi(0)}{|\partial_C\varphi(0)|} + f(w_0)   
\]
whose ``radius'' is at least
\[
\varrho=\lim_{r\to 1}\varrho(r)\geq \lim_{r \to 1}\frac{r}{32\sqrt{2}}=\frac{1}{32\sqrt{2}}.
\]
%\qed  
\end{proof}

\noindent It is easy to prove that if the regular translation $\tilde f_u$ that appears in the statement of Theorem \ref{BlochLandau} is a real translation (i.e. if  $u$ is real), then the universal set  $\mathcal{O}(\frac{1}{32\sqrt{2}})$ is contained in $f(\mathbb{B})$.

\noindent It seems to us that, in general,  there is  not a universal open set directly contained in the image $f(\mathbb{B})$ of a (normalized) slice regular function $f$. And this might be connected with the fact that, as proven in \cite{Duren}, the Bloch-Landau Theorem does not hold in $\mathbb{C}^2$ (and $\mathbb{C}^2$ and $\mathbb{H}$ are strictly related).
 
\noindent In any case the authors plan to further investigate this fascinating subject in the near future.

%\noindent As a consequence of the remarks made at the beginning of this section, we have also the statement in the ``classical version''.
%\begin{coro}
%Let $f:\B \rightarrow \HH$ \ be a regular function such that $f(0)=0$ and $\partial_Cf(0)=1$. Then there exists $u \in \B$ such that the image of $\tilde{f}_u$ contains an open ball whose radius is at least $\frac{37}{2^{19}}$.
%\end{coro}

\end{document}